\documentclass[12pt]{amsart}

\usepackage[T1]{fontenc}
\usepackage[cp1250]{inputenc}
\usepackage[margin=1in]{geometry}
\usepackage{color}
\usepackage[centertags]{amsmath}
\usepackage{amsfonts}
\usepackage{amsthm}
\usepackage{newlfont}
\usepackage{amscd}
\usepackage{verbatim}
\usepackage{amsgen}
\usepackage{amssymb}
\usepackage{amssymb,amsmath}
\usepackage{amscd}
\usepackage{enumerate}
\usepackage{multicol}
\usepackage{tikz-cd}
\DeclareMathOperator{\Hilb}{Hilb}

\DeclareMathOperator{\len}{len}
\newcommand{\gota}{\mathfrak{a}}

\newlength{\defbaselineskip} 
\theoremstyle{plain}
\newtheorem{thm}{Theorem}[section]
\newtheorem{cor}[thm]{Corollary}
\newtheorem{con}[thm]{Conjecture}
\newtheorem{df}[thm]{Definition}
\newtheorem{lema}[thm]{Lemma}
\newtheorem{obs}[thm]{Proposition}
\newtheorem{exm}[thm]{Example}

\newtheorem{fact}[thm]{Fact}

\newtheorem{rem}[thm]{Remark}
\newtheorem{pr}{Algorithm}
\newtheorem{pro}{Main Problem}
\theoremstyle{definition} \newtheorem{ex}{Example}[section] %

 \numberwithin{equation}{section}

\def\p{\mathbb{P}}
\def\r{\mathbb{R}}

\def\n{\mathbb{N}}

\def\c{\mathbb{C}}

\def\C{\mathbb{C}}

\def\fa{\begin{fact}}
\def\kfa{\end{fact}}

\newcommand{\fromto}[2]{#1, \dotsc, #2}

\newcommand{\set}[1]{\left\{#1\right\}}

\DeclareMathOperator{\Spec}{Spec}

\def\p{\mathbb{P}}
\def\a{\mathbb{A}}
\def\ob{\begin{obs}}
\def\kob{\end{obs}}
\def\dow{\begin{proof}}
\def\kdow{\end{proof}}

\def\tw{\begin{thm}}
\def\ktw{\end{thm}}
\def\hip{\begin{con}}
\def\khip{\end{con}}
\def\lem{\begin{lema}}
\def\klem{\end{lema}}
\def\ex{\begin{exm}}
\def\prog{\begin{pr}}
\def\kprog{\end{pr}}
\def\wn{\begin{cor}}
\def\kwn{\end{cor}}
\def\uwa{\begin{rem}}
\def\kuwa{\end{rem}}
\def\kex{\end{exm}}
\def\dfi{\begin{df}}
\def\kdfi{\end{df}}
\definecolor{zielony}{rgb}{0.5, 0.9, 0.1}
\definecolor{czerwony}{rgb}{0.9, 0.2, 0.1}
\definecolor{niebieski}{rgb}{0.3, 0.1, 0.9}

\newtheorem{lemma}[thm]{Lemma}

\newtheorem{example}[thm]{Example}

\theoremstyle{definition}

\begin{document}
\title{Examples of $k$-regular maps and interpolation spaces}
\author{Mateusz Micha\l ek}
\address{
University of California, Berkeley\\
Department of Mathematics\\
951 Evans Hall, Suite 968\\
Berkeley, CA 94720-384 \newline
Freie Universit\"at\\
 Arnimallee 3\\
 14195 Berlin, Germany\newline
Polish Academy of Sciences\\
         ul. \'Sniadeckich 8\\
         00-956 Warsaw\\
         Poland}
\email{wajcha2@poczta.onet.pl}
\author{Christopher Miller}
\address{University of California, Berkeley\\
Department of Mathematics\\
951 Evans Hall, Suite 3840\\
Berkeley, CA 94720-384 }
\email{crmiller@math.berkeley.edu}
\thanks{The work was supported by a grant Polish National Science Center grant no. 2012/05/D/ST1/01063}

\begin{abstract}
A continous map $f: \C^n \rightarrow \C^N$ is $k$-regular if the image of any $k$ points spans a $k$-dimensional subspace. It is an important problem in topology and interpolation theory, going back to Borsuk and Chebyshev, to construct $k$-regular maps with small $N$ and only a few nontrivial examples are known so far. Applying tools from algebraic geometry we construct a 4-regular polynomial map $\C^3\rightarrow \C^{11}$ and a 5-regular polynomial map $\C^3\rightarrow \C^{14}$.
\end{abstract}

\maketitle

\section{Introduction}
The motivation for our work comes from topology and interpolation theory. Let $X$ be a real manifold. The following definition and problem are due to Borsuk \cite{Borsuk}.
\begin{df}[$k$-regular map]
A map $f: X \rightarrow \r^N$ is $k$-regular if the images of any $k$ points are linearly independent. 
\end{df}
\begin{pro}\label{pr1}
Given $n$ and $k$ what is the smallest possible $N$ such that there exists a $k$-regular map  $\r^n \rightarrow \r^N$?
\end{pro}
We also have the following definition and interpolation problem going back to Chebychev, Haar \cite{haar__kreg} and Kolmogorov \cite{kolmogorov__kreg}.
\begin{df}[$k$-interpolating space]
A vector space $V$ of continuous functions on $X$ is called $k$-interpolating if for any distinct points $P_1,\dots,P_k\in X$ and any scalars $\lambda_1,\dots,\lambda_k\in \r$ there exists a function $f\in V$ such that $f(P_i)=\lambda_i$ for all $i=1,\dots,k$.
\end{df}
\begin{pro}\label{pr2}
Given $X=\r^n$ and $k$ what is the smallest possible dimension of a $k$-interpolating space?
\end{pro}
Problems \ref{pr1} and \ref{pr2} are classically known to be equivalent.
\tw\label{tw:equivalence}
There exists an $N$-dimensional $k$-interpolating space $V$ if and only if there exists a $k$-regular map $f:\r^n \rightarrow \r^N$. 
\ktw
For the sake of completeness we provide a short and easy proof in Section \ref{sec:defandpre}. Both problems attracted attention of many mathematicians throughout the years. Just to name a few: P.~Blagojevi{\'c}, M.~Chisholm, F.~Cohen, D.~Handel, W.~L{\"u}ck, J.~Segal, B.~Shekhtman V.~A.~Vassiliev, G.~Ziegler \cite{Ziegler, kreg1, kreg2, kreg3, kreg4, kreg5, interpolation__Shekhtman__poly, kreg6}.
Although the problems were originally posed in the fifties over $\r$ one can, and indeed one does \cite{BJJM, Ziegler2}, address similar questions on $\C$. In general, it is easier to construct $k$-regular maps on $\r$. Indeed, if a $k$-regular map on $\C$ takes real values on $\r$, then it is also $k$-regular on $\r$. On the other hand, even a polynomial $k$-regular map on $\r$ does not have to be $k$-regular on $\C$. Indeed:
$$\r^n\ni(x_1,\dots,x_n)\rightarrow (1,x_1,\dots,x_n,x_1^2+\dots+x_n^2)\in\r^{n+2}$$
is $3$-regular. However, there exists a $3$-regular map $\C^n\rightarrow \C^N$ if and only if $N\geq 2n+1$. 

The techniques applied so far to construct interpolating spaces rely on very clever constructions of polynomials with special properties \cite[Example 2.6]{Ziegler},\cite{interpolation__Shekhtman__poly}. Our techniques are very different and work over the complex numbers. We build on methods introduced in \cite{BJJM}. The construction starts from a polynomial map obtained as a projection from a sufficiently high Veronese embedding. The center of the projection has to be chosen in such a way that it does not intersect an areole - a local analog of a secant variety, cf.~Definitions \ref{def:sec}, \ref{def:aur}. These techniques are in general nonconstructive and nonpolynomial. The main aim of this paper is to present explicit examples where we can obtain interpolating spaces spanned by polynomials.
\tw
The map:
$$\C^3\ni(s,t,u) \mapsto (1,s,t,u,st,su,tu,s^2,s^3-t^2,t^3-u^2,u^3)\in \C^{11}$$
is $4$-regular.
The map:
$$\C^3\ni(s,t,u) \mapsto (1,s,t,u,st,su,tu,s^2,t^2,u^2,u^3,t^4-s^3,u^4-t^3,s^4)\in \C^{14}.$$
is $5$-regular.
\ktw
In particular, the first result improves \cite{interpolation__Shekhtman__poly}. Our construction of such maps relies on the study of the Gorenstein locus of the punctual Hilbert scheme \cite{MR3404648}. Intuitively, we first consider an algebraic case when the $k$ points are arbitrarily close. Then we use a torus action to pass from a local to global construction. Our results are polynomials which is important from the point of view of applications. Indeed, some authors define interpolating spaces to be spanned by polynomials \cite[Definition 6.5.2]{interpolation__Shekhtman__ideal}. 

There are two main difficulties in solving Main Problem \ref{pr1}. First, it is hard to show that a given map is $k$-regular. Second, even finding polynomial examples is in general difficult. The following result, essentially due to Alexander and Hirshowitz \cite{MR1311347}, shows that generic polynomials will not provide useful examples. In particular, a map will be $k$-regular only if the polynomials are chosen in a very special way. 

\tw\label{th:generic}
Generic polynomials $f_1,\dots, f_N\in \C[x_1,\dots,x_n]$ of degree $d\geq k-1>1$ define a $k$-regular map $\C^n\rightarrow \C^N$ if and only if $N\geq (n+1)k-1$, unless:
\begin{enumerate}
\item $n=1$ in which case $N\geq k-1$,
\item $n=2$, $k=4$ in which case $N\geq 9$,
\item $n=2$, $k=5$ in which case $N\geq 13$,
\item $k=3$ in which case $N\geq 3n-1$.
\end{enumerate}
\ktw
For the sake of completeness we sketch the proof in Section \ref{sec:defandpre}.

The very-well developed techniques to prove that for given $n$ and $N$ there is no $k$-regular map rely on tools from algebraic topology. It would be of great interest to obtain sharper results for polynomial $k$-regular maps. 
Currently, it is not known if there exists a continuous $k$-regular map $\C^n \rightarrow
\C^N$, such that there is no polynomial $k$-regular map $\C^n \rightarrow \C^N$.



\section{Definitions and Preliminaries}\label{sec:defandpre}




We begin this section with a proof of Theorem  \ref{tw:equivalence}. We will then go over the important definitions.

\begin{proof}[Proof of Theorem \ref{tw:equivalence}]
We begin with the forward direction. Pick a basis for $V = \langle f_1,\ldots, f_N\rangle$. Let $f= (f_1,\ldots f_N) : \r^n \rightarrow \r^N$. Take any $x_1,\ldots x_k \in \r^n$ and any $\beta_i \in \r$ such that 
$$\sum_{i=1}^k \beta_i f(x_i) =0,\text{ i.e. }\sum_{i=1}^k \beta_i f_j(x_i) =0 \text{ for all } j.$$ 
Multiplying each equation by arbitrary $\lambda_j$, we can write

$$\sum_{i=1}^k \beta_i [\sum_{j=1}^N \lambda_j f_j](x_i) =0.$$

However, $\sum_{j=1}^N \lambda_j f_j$ is an arbitrary element of $V$. Taking $\sum_{j=1}^N \lambda_j f_j = g_l\in V$ to be the function which assigns $g_l(x_i) = \delta_{li}$, we get $\beta_i =0$ for each $i$. Hence $\{f(x_i)\}$ is linearly independent, and $f$ is $k$-regular. 

For the reverse direction define $V = \{L \circ f \mid L: \mathbb{R}^N \rightarrow \mathbb{R} \text{ linear} \}$. We can obviously assign arbitrary values to linearly independent elements, i.e.~$V$ is $k$-interpolating.
\end{proof}

A classical example of a $k$-regular map is the $r$-th Veronese embedding $v_r$ for $r\geq k-1$. Geometrically it is an embedding (or its restriction to an open affine subset) of $\p^n$ in $\mathbb{P}^{\binom{n+r}{r}-1}$ given by the very ample line bundle $\mathcal{O}(r)$. In other terms, the map associates to a class of a vector the class of its $r$-th symmetric power.
Explicitly in coordinates, $v_r$ is given by all homogeneous monomials of degree $r$. 

Restricting to the affine space $\a^n=\{[x_0,\dots,x_n]\in\p^n: x_0 =1\}$, we get a map $\a^n \rightarrow \a^{\binom{n+r}{n}-1}$. Explicitly it is given by all monomials in $x_1,\dots,x_n$ of degree at most $r$.

It is a well-known fact that if $r\geq k-1$, then for any points $P_1,\dots,P_k\in\a^n$ their images $v_r(P_i)$ are linearly independent. It follows that for $r\geq k-1$ the map $v_r:\a^n\rightarrow \a^{\binom{n+r}{n}-1}$ is $k$-regular.
This has a very nice generalization in scheme theoretic language, that played a crucial role in Grothendieck's construction of the Hilbert scheme. The Veronese map is also universal, in the sense that any map given by (nonhomogenoues) polynomials of degree (at most) $r$ is given by composing $v_r$ with a projection. 
\lem
Let $f$ be any map $f:\a^n \overset{v_r}{\rightarrow} \a^{\binom{n+r}{n}-1} \overset{\pi}{\rightarrow} \a^N$ where $\pi$ has center $E$. Then $f$ is $k$-regular if and only if for all distinct $x_1,\dots,x_k\in v_r(\a^n)$ the vector space  $\langle x_1,\ldots ,x_k \rangle$ intersects the vector space $E$ only at $0$. 
\klem
\dow 
The lemma follows from the fact that the image of any vector subspace $V_1$ under a linear projection with center $E$ is a vector space of dimension $\dim V_1-\dim (V_1\cap E)$.
\kdow
The previous Lemma shows that the following geometric object will be of special importance for us.
\dfi[Secant variety]\label{def:sec}
  For any embedded variety $X$ we define its \textbf{$k$-th secant variety} as
   \begin{align*}
     \sigma_k(X)&:= \overline{\bigcup_{\fromto{x_1}{x_k}\in X}  \langle \fromto{x_1}{x_k}\rangle},
   \end{align*}
   where $<\cdot>$ is a linear span in either projective or affine space.
\kdfi
Let us prove Theorem \ref{th:generic}.
\dow 
A generic map of polynomials corresponds to a composition of a Veronese map and a generic linear projection. A generic vector space $E$ will intersect some linear space spanned by $v_r(x_1),\dots,v_r(x_k)$ if and only if $\dim E$ is greater than the codimension of the union $\bigcup_{x_i\in\a^n} <0,v_r(x_1),\dots,v_r(x_k)>$. By Alexander-Hirshowitz Theorem the dimension of the latter equals $kn+k-1$ unless we are in one of the cases specified in the theorem. We refer to \cite[Theorem 6.6 and Corollary 6.7]{BJJM} for more details.
\kdow

Recall that an affine scheme $S$ can be identified with an ideal $I_S\subset \C[x_1,\dots,x_n]$. A scheme is, by definition, finite if $\C[x_1,\dots,x_n]/I_S$ is a finite dimensional vector space (over $\C$). We call $\len S:=\dim_{\C}  \C[x_1,\dots,x_n]/I_S$ the \emph{length} of $S$. If $I_S$ consists of all polynomials vanishing at fixed $k$ distinct points, then the length of $S$ equals $k$. Such finite schemes are called smooth. The vanishing locus of polynomials in $I_S$ is called the support of $S$.
There exist schemes of length $k$ with support of smaller cardinality. In particular, if the support is just one point we call the scheme \emph{local}\footnote{The corresponding algebra $\C[x_1,\dots,x_n]/I_S$ has only one maximal ideal and is also called local.}. 
\begin{example}
The scheme corresponding to the ideal $(x^3)\subset \C[x]$ is a finite local scheme of length $3$. The unique maximal ideal is $m=(x)\subset \C[x]/(x^3)$. 
\end{example}
Recall that a linear span of a set is the zero locus of all linear polynomials that vanish at this set. This motivates the definition of the linear span of any scheme.
\dfi[Linear Span of a Scheme]
The linear span of the any scheme $S$ is defined to be the zero locus of all linear equations in the ideal of $S$:
$$<S>:=\{x:l(x)=0\text{ for all }l\in I_S,\deg l=1\}.$$
\kdfi
In general $\dim <S>\leq \len S-1$. Moreover, for $r\geq \len S-1$ we have $\dim <v_r(S)>=\len S-1$. This allows us to identify the scheme $S$ with the point in  the Grassmannian $Gr(\len S-1,{n+r\choose r})$. Let $U$ be the image in $Gr(k-1,{n+r\choose r})$ of all smooth schemes of length $k$. The points in the closure of $U$ also correspond to schems of length $k$ called \emph{smoothable}, i.e.~limits of smooth schemes. The closure of $U$ is called the smoothable component of the Hilbert scheme and we will denote it by $\Hilb^{sm}$.

We present the following two definitions that are punctual variants of the secant variety.
\dfi[Areole]\label{def:aur}
Our first definition was introduced in \cite{BJJM}.
Let $X$ be a variety and $p\in X$ any point.
   The \textbf{$k$-th areole} at $p$ is
   \begin{align*}
     \gota_k(X,p)&:= \overline{ \bigcup \set { \langle R \rangle \mid   
                              R\text{ is smoothable in }X \text{, supported at } p \text{ and } \len(R) \leq k } }.
   \end{align*}
    We also define
   \begin{align*}
     \tilde\gota_k(X,p)&:= \overline{ \bigcup \set {\lim <x_1,\dots,x_k>: x_i\in X\text{ and }\lim x_i=p} },
   \end{align*}
   i.e.~ the union is taken over all $k$-tuples of points in $X$ converging to $p$. In other words to construct $\gota_k(X,p)$ we take the linear spans of all possible limits of $k$-points converging to $p$. To construct $\tilde\gota_k(X,p)$ we take the limits of linear spans.
\kdfi
In general, the linear span of the limit scheme is contained in the limit of the linear spans, hence $\gota_k(X,p)\subset\tilde\gota_k(X,p)$. However, in our case, i.e.~when $X$ is a high Veronese embedding of an affine or projective space, the limit of linear spans and the linear span of the limit are of the same dimension, hence must be equal. In particular, $\gota_k(X,p)=\tilde\gota_k(X,p)$. For similar reasons, in our cases, the closures in the definitions of areoles are not needed - the union is already closed.
We introduce $\tilde\gota_k(X,p)$ as recently it turned out to be of interest in the study of complexity of matrix multiplication algorithms \cite{JaLandsberg}. Moreover, although not appearing explicitly, the variety $\tilde\gota_k(X,p)$ played an important role in the proof of the following theorem.
\tw\label{tw:avoidareole}
Consider the Veronese map $v_r:\a^n\rightarrow \a^{n+r\choose r}$ composed with a projection $\pi:\a^{n+r\choose r}\rightarrow \a^N$. If the center $E$ of projection $\pi$ is disjoint from the areole $\gota_k(v_r(\a^n),v_r(P))$ for some point $P\in \a^n$, then $\pi\circ v_r$ is $k$-regular in some open neighborhood of $P$.
\ktw
\dow 
This is a special case of \cite[Lemma 3.2]{BJJM}
\kdow
Unfortunately providing a precise description of the areole is hard, as classifying local smoothable schemes is difficult. One idea is to pass to a different class of schemes - Gorenstein schemes - that turn out to be sufficient and easier to classify.

There are several equivalent definitions of Gorenstein schemes. The one below we present is in fact a theorem known as Macaulay duality. We chose it as it is both simple and useful in classification.
\dfi[Gorenstein local scheme/ Macaulay duality, \cite{iarrobino_kanev_book_Gorenstein_algebras},Chapter 2.3] 
We say that a local scheme $S$ is Gorenstein if and only if its ideal equals the ideal of differentials that annihilate some polynomial $f_S$. We call $f_S$ the dual socle generator. 
\kdfi
\ex 
All schemes $\C[x]/(x^k)$ are Gorenstein, as $(\partial^k/\partial^ky)$ is the ideal annihilating $y^{k-1}$. Also $\C[x_1,x_2]/(x_1x_2,x_1^2-x_2^2)$ is Gorenstein, as it is the ideal of annihilators of $y_1^2+y_2^2$. 

On the other hand $\C[x_1,x_2]/(x_1^2,x_1x_2,x_2^2)$ is not Gorenstein. Indeed, the hypothetical dual socle generator would have to be of degree at most $1$, as all differentials of order at least two would have to annihilate it. However, than there would also exist a differential operator of order $1$ that annihilates it.
\kex
Gorenstein schemes are important for us due to the following result.
\lem\cite[Lemma 2.3]{BB14}
Let $R$ be a finite scheme (embedded in $\a^n$). Then
\[
  \langle R \rangle = \bigcup_{\begin{array}{c}
                                 Q \subset R \\ Q \text{ is Gorenstein}
                               \end{array}}
                                \langle Q \rangle.
\]
\klem
Hence, to avoid the areole it is enough to avoid linear spans of Gorenstein schemes. 
In the classification of Gorenstein schemes, the Hilbert function plays a crucial role.
\dfi[Hilbert Function]
Consider a finite local scheme $S$ with the associated ideal $I_S\subset \C[x_1,\dots,x_n]$, with $m=(x_1,\dots,x_n)$ the unique maximal ideal in $ \C[x_1,\dots,x_n]/I_S$. 
We define a Hilbert function $H_S:\n\rightarrow \n$ by:
$$H_S(i)=\dim_\C m^i/m^{i+1}.$$
\kdfi
\ob\label{properties}
Using notation as above Hilbert function has the following properties.
\begin{enumerate}
\item If $H_S(i)=0$ then $H_S(j)=0$ for all $j\geq i$. This follows immediately, as $H_S(i)=0$ if and only if $m^i=m^{i+1}$.
\item $H_S(0)=1$, as $\C[x_1,\dots,x_n]/m\simeq \C$.
\item $H_S$ has nonzero values only on a finite interval $[0,\dots,d]$, as $S$ is assumed finite.
\item $\sum_i H_S(i)=\len S$.
\item If $S$ is Gorenstein we have $H_S(d)=1$ (the last nonzero value), consult e.g.~\cite[2.3]{iarrobino_kanev_book_Gorenstein_algebras}.
\item If $H_S(1)=1$, then $H_S(i)\leq 1$, as $S$ can be represented by in ideal in $\C[x]$.
\end{enumerate}
\kob

\section{Explicit examples of polynomial $k$-regular maps}


To obtain a $k$-regular map we first take the $k-1$st Veronese embedding and then project from subspaces disjoint from the areola. We start by showing that we can always project away from monomials of highest degree that are not pure powers.

\begin{lemma}
Given the Veronese embedding $v_{k-1}: \mathbb{A}^n \rightarrow \mathbb{A}^{\binom{n+k-1}{k-1}}$, pick any coordinate with corresponding monomial of degree $k-1$ that is not a pure power. The line defined by this coordinate is disjoint from the areole $\gota_{k-1}(v_{k-1}(\a^n),v_{k-1}(0))$.
\end{lemma}

\begin{proof}
Pick a mixed monomial $m$ of degree $k-1$ and let $L_m$ be the linear function on $\mathbb{A}^{\binom{n+k-1}{k-1}}$ corresponding to the given coordinate. Our claim is equivalent to showing that a point $P$ that has all coordinates equal to zero, but $L_m(P)\neq 0$, does not belong to $\gota_{k-1}(v_{k-1}(\a^n),v_{k-1}(0))$. We will show that for any local scheme $S$ of length $k$ supported at $0\in \a^n$ we have a point $P\not\in <v_{k-1}(S)>$. This last claim will be established by providing a linear form $L_S$ that vanishes on $<v_{k-1}(S)>$, but $L_S(P)\neq 0$.  
We have:
\begin{center}
\begin{tikzcd}
\C[L_1,L_{x_1},L_{x_2},\ldots ,L_{{x_1}^{k-1}},\dots,L_m,\dots,L_{{x_n}^{k-1}}] \arrow{d}{v}\\ \C[x_1,\ldots ,x_n] \arrow{r}{\mathfrak{i}} & \C[x_1,\dots ,x_n]/I_S \\
\end{tikzcd}
\end{center}
where $v$ corresponds to Veronese, sending the variable $L_{\tilde m}$ to $\tilde m$ for any monomial $\tilde m$. The map $\mathfrak{i}$ corresponds to the embedding of $S$ in $\a^n$. The vanishing of any linear form $L$ on $v_{k-1}(S)$ is by definition equivalent to the fact $L$ is in $\ker(\mathfrak{i} \circ v)$. If $H_S(k-1)=0$ then, by degree count, $L_m$ is in $\ker(\mathfrak{i} \circ v)$. Hence we only have to consider schemes such that $H_S(k-1)\neq 0$. By Lemma \ref{properties}, as $\len S=k$, we must have $H_S(i)=1$ for $i<k$ and $H_S(i)=0$ for $i\geq k$. There is only one (up to isomorphism) scheme with this Hilbert function: $\Spec \C[x]/(x^k)$. Hence we may replace $\mathfrak{i}$ by $ \alpha:\C[x_1,\dots,x_n]\rightarrow \C[x]/(x^k)$. 


Now $(\alpha \circ v)(L_m) =c x^{k-1}$ for some $c\in\C$. If $c=0$ we are done, so we assume $c\neq 0$.
Pick any variable $x_i|m$. We know that the image of $L_{x_i}$ equals $\sum_{j=1}^{k-1} a_jx^i$. We have $a_1\neq 0$, as $c\neq 0$. Then the linear form $L_m-\frac{c}{a_1^{k-1}}L_{x_i^{k-1}} \in \ker(\alpha \circ v)$ is nonzero on $P$.
\end{proof}
By Theorem \ref{tw:avoidareole}, we can project from the line corresponding to a given mixed monomial of highest degree, and maintain local $k$-regularity. Moreover, in the proof of the lemma we only used $L_{\tilde m}$ where either $\tilde m=m$ or $\tilde m$ is a pure power. Hence, we may in fact project away from all mixed powers of highest degree.

\begin{lemma}\label{lem:action}
Let $f = (f_1,f_2,\ldots f_N) :\C^n \rightarrow \C^N$ be a polynomial map that is $k$-regular on some open neighborhood of $0\in \C^n$.  If there exists a positive weighting of the $x_i$ such that each $f_i$ is homogeneous, then the map is globally $k$-regular.
\end{lemma}


\begin{proof}
Pick $\epsilon$ sufficiently small so that $f$ is $k$ regular on the ball $B_\epsilon(0)$. Suppose for contradiction that $f$ is not globally $k$-regular. Then there exists $\{x^i = (x^i_1,x^i_2,\ldots x^i_n)\}_{i=1}^k$ , a set of $k$ points such that the image is not linearly independent. Here the exponent is just an index, and does not indicate any operation. We know that the matrix
\begin{center}
$
\left |
\begin{array}{c c c}
f_1(x^1) & \cdots &f_N(x^1) \\
\vdots & & \\
f_1(x^k) & \cdots & f_N(x^k) \\
\end{array}
\right |
$
\end{center}
does not have maximal rank. So all of its maximal minors vanish. 

Let $w: \{x_1, \ldots, x_n\} \rightarrow \mathbb{N}$ be a positive weighting such that  each $f_i$ is homogeneous of degree $d_i$. Then $f_i(\lambda^{w(x_1)}x_1 ,\ldots , \lambda^{w(x_n)}x_n) = \lambda^{d_i}f_i(x_1,\ldots x_n)$. Define $\tau_\lambda(x^i)$ to be the point $(\lambda^{w(x^i_1)}x^i_1 ,\ldots , \lambda^{w(x^i_n)}x^i_n)$. This function describes the {\em torus action}, which here scales each coordinate by some positive power of $\lambda$. Then we have the equality

\begin{center}
$
\left |
\begin{array}{c c c}
f_1(\tau_\lambda(x^1)) & \cdots &f_N(\tau_\lambda(x^1)) \\
\vdots & & \\
f_1(\tau_\lambda(x^k)) & \cdots & f_N(\tau_\lambda(x^k)) \\
\end{array}
\right |
$
 =
 $
\left |
\begin{array}{c c}
\lambda^{d_1}f_1(x^1) & \cdots \lambda^{d_N}f_N(x^1) \\
\vdots & \\
\lambda^{d_1}f_1(x^k) & \cdots  \lambda^{d_N} f_N(x^k) \\
\end{array}
\right | .
$
\end{center}

All of the maximal minors will still vanish, because we have only scaled the column vectors. So the matrix doesn't have maximal rank, and the row vectors are linearly dependent.

Taking $\lambda$ to be sufficiently small, we have $\tau_\lambda(x^i) \in B_\epsilon(0)$ for all $i \leq k$. But the span of the images of these points is not maximal dimensional, so $f$ is not $k$-regular on $B_\epsilon(0,\ldots,0)$ --- a contradiction.
\end{proof}
\wn\label{wn:1kreg}
The map $\C^n\rightarrow \C^{{n+k-2 \choose k-2}+n}$ given by all monomials of degree at most $k-2$ and $n$ monomials $x_i^{k-1}$ is $k$-regular.
\kwn

\subsection{$\C^3\rightarrow \C^{11}$ 4-regular}
In this section we prove the following theorem.
\tw\label{tw:3}
For any $k>3$ the map:
$$\C^n\rightarrow \C^{{n+k-2 \choose k-2}+1}$$
given by:
\begin{itemize}
\item  All $n+k-3 \choose n$ monomials of degree at most $k-3$,
\item All ${n+k-3 \choose n-1} -n$ monomials of degree $k-2$ that are not pure powers,
\item $n-1$ polynomials $x_{i+1}^{k-1}-x_i^{k-2}$ for $i=1,\dots,n-1$,
\item $x_1^{k-1}$ and $x_n^{k-2}$
\end{itemize}
is $k$-regular. 
\ktw
\wn
For $k=4$ and $n=3$ we obtain a $4$-regular map:
$$\C^3\ni(s,t,u) \mapsto (1,s,t,u,st,su,tu,s^2,s^3-t^2,t^3-u^2,u^3)\in \C^{11}.$$
\kwn
\dow[Proof of Theorem \ref{tw:3}] 
We start with the $k$-regular map constructed in Corollary \ref{wn:1kreg}. Let $P_j$ be the point whose coordinates are zero, except for those coordinates corresponding to $x_j^{k-1}$ and $x_{j+1}^{k-2}$, which are nonzero and equal. We can project consecutively from $P_j$, with $j=1,\dots,n-1$.  In other words we replace the two monomials in the map by their difference. 

Our fist claim is that $P_j$ does not belong to the $k$-th areole. This, by Theorem \ref{tw:avoidareole}, will show that after the projection we obtain a map that is $k$-regular in some neighborhood of zero. This will finish the proof, as all the components of our map are monomials, apart from $n-1$ binomials $x_j^{k-1}-x_{j+1}^{k-2}$. Hence Lemma \ref{lem:action} applies with the weight $\omega(x_j)=(k-1)^{j-1}(k-2)^{n-j}$.

To prove that $P_j$ does not belong to the areole we fix an embedding of a scheme $S$ supported at zero and consider its image under the $k$-regular map constructed so far:

\begin{center}
\begin{tikzcd}
\C[L_{\tilde m}] \arrow{d}{v}\\ \C[x_1,\ldots ,x_n] \arrow{r}{\mathfrak{i}} & \C[x_1,\dots ,x_n]/I_S \\
\end{tikzcd}
\end{center}

If $H_S(k-1)=0$, the linear form $L_{x_j^{k-1}} \in \ker( \mathfrak{i} \circ v)$ satisfies $L_{x_j^{k-1}}(P_j)\neq 0$. Hence $P_j$ is not in the linear span of such $S$.

If $H_{S}(k-1)\neq 0$, by Properties \ref{properties} we must have $H_{S}(j)=1$ for $j\leq k-1$ and $H_S(j)=0$ for $j\geq k$. Hence the algebra of $S$ equals $\c[x]/(x^k)$ and we obtain:

\begin{center}
\begin{tikzcd}
\C[L_{\tilde m}] \arrow{d}{v}\\ \C[x_1,\ldots ,x_n] \arrow{r}{\alpha} & \C[x]/(x^k) \\
\end{tikzcd}
\end{center}

Let us consider the form $L_{x_{j+1}}$ that is mapped to $ (\mathfrak{i} \circ v )(L_{x_{j+1}}) = \sum_{i=1}^{k-1} a_ix^i$. If $a_1=0$, we know that $L_{x_{j+1}^{k-2}}$ is in $\ker (\mathfrak{i} \circ v)$ (as $k>3$). Since $L_{x_{j+1}^{k-2}}(P_j)\neq 0$, this finishes the proof in this case. 

Let $\sum_{i=1}^{k-1} b_ix^i$ be the image of $x_{j}^{k-1}$ under the composition of the above maps.
If $a_1\neq 0$ we consider the linear form $L_{x_j^{k-1}}-\frac{b_1}{a_1^{k-1}} L_{x_{j+1}^{k-1}}$. The linear form clearly belongs to $\ker (\mathfrak{i} \circ v)$  and is nonzero on $P_j$.
\kdow

\subsection{$\C^3\rightarrow \C^{14}$ 5-regular}

\tw\label{tw:4}
For any $k>4$ the map:
$$\C^n\rightarrow \C^{{n+k-3 \choose n}+n+1}$$
given by:
\begin{itemize}
\item  All $n+k-3 \choose n$ monomials of degree at most $k-3$,
\item $n-1$ polynomials $x_{i+1}^{k-1}-x_i^{k-2}$ for $i=1,\dots,n-1$,
\item $x_1^{k-1}$ and $x_n^{k-2}$
\end{itemize}
is $k$-regular. 
\ktw

\wn
For $k=5$ and $n=3$, we obtain a $5$-regular map:
$$\C^3\ni(s,t,u) \mapsto (1,s,t,u,st,su,tu,s^2,t^2,u^2,u^3,t^4-s^3,u^4-t^3,s^4)\in \C^{14}.$$

\kwn

\dow[Proof of Theorem \ref{tw:4}] 
We start with the $k$-regular map constructed in Corollary \ref{wn:1kreg}. We let $m$ be a mixed monomial of degree $k-2$, and let $P_m$ be the point with the only nonzero entry corresponding to $m$. 

To prove that $P_m$ does not belong the the areole we fix an embedding of a scheme $S$ supported at zero and consider its image under the $k$-regular map constructed so far:

\begin{center}
\begin{tikzcd}
\C[L_{\tilde m}] \arrow{d}{v}\\ \C[x_1,\ldots ,x_n] \arrow{r}{\mathfrak{i}} & \C[x_1,\dots ,x_n]/I_S \\
\end{tikzcd}
\end{center}

If $H_S(k-2)=0$ the linear form $L_{m}$ is in the kernel of the composed maps and $L_{m}(P_m)\neq 0$. Hence $P_m$ is not in the linear span of such $S$.

If $H_{S}(k-2)\neq 0$, by Properties \ref{properties}, our Hilbert function must be one of the following:



\begin{align*}
(1^k,0,\dots),(1^{k-1},0,\dots), (1,2,1^{k-3},0,\dots)
\end{align*}
where $1^a$ denotes a sequence of $a$ consecutive ones.

The first Hilbert sequence describes a scheme with the embedding $\alpha$ below, up to isomorphism.

\begin{center}
\begin{tikzcd}
\C[L_{\tilde m}] \arrow{d}{v}\\ \C[x_1,\ldots ,x_n] \arrow{r}{\alpha} & \C[x_1,\dots ,x_n]/(x^k) \\
\end{tikzcd}
\end{center}

Every divisor $x_i | m$ is sent to a term $\alpha (x_i) = a_{i1}x^1 + a_{i2}x^2+ \cdots$. If $a_{i1}=0$ for all $i$, then $\alpha(m$) is degree at least $2(k-2) \geq k$. So $\alpha(m)=0$ , and we can project from  $P_m$. Thus we assume that $a_{i1}$ is nonzero for at least one value of $i$. Then 

\begin{align*}
(\alpha \circ  v) (L_{x_i^{k-2}}) &=  a_{1i}^{k-2}x^{k-2} + (k-2)a_{1i}^{k-3}a_{i2}x^{k-1}\\ 
(\alpha \circ  v) (L_{x_i^{k-1}}) &=  a_{1i}^{k-1}x^{k-1}
\end{align*}

If we write $(\alpha \circ  v) (L_m) = bx^{k-2} + cx^{k-1}$, then the linear form 

$$L_m - \frac{b}{a_{i1}^{k-2}}L_{x_i^{k-2}} - \frac{c-\frac{b}{a_{i1}}(k-2)a_{i2}}{a_{i1}^{k-1}}L_{x_i^{k-1}}$$

does not vanish at $P_m$, and is in $\ker (\alpha \circ  v)$. This finishes the proof in the first case.\\

The second case 
%
is exactly as in the proof of Theorem \ref{tw:3}. For the final case, by \cite[Theorem 2.9]{cn09}, our diagram must have the form

\begin{center}
\begin{tikzcd}
\C[L_{\tilde m}] \arrow{d}{v}\\ \C[x_1,\ldots ,x_n] \arrow{r}{\beta} &\C[x,y]/(xy,x^{k-2}-y^2) \\
\end{tikzcd}
\end{center}

The image of $L_m$ in this algebra is a homogeneous polynomial of degree $k-2 \geq 3$, i.e. $(\beta \circ v) (L_m)= cx^{k-2}$. Let $x_i $ be a divisor of $m$. Then $(\beta \circ v) (L_{x_i}) = ax+by+\cdots$. If $a=0$, then $c=0$ in $(\beta \circ v) (L_m)$, and we can project from $P_m$. 

If we assume that $a \neq 0$, then $L_m - \frac{c}{a^{k-2}}L_{x_i^{k-2}}$ belongs to the kernel of composition and is nonzero at $P_m$.

Thus we have projected from all coordinates corresponding to mixed monomials of degree $k-2$. We now proceed as in Theorem \ref{tw:3}, mapping from coordinates that are zero except at coordinates corresponding to $x_j^{k-1}$ and $x_{j+1}^{k-2}$. This map is $k$-regular on some open neighborhood around the origin. We use the weight $\omega(x_j)=(k-1)^{j-1}(k-2)^{n-j}$, and by Lemma \ref{lem:action}, the map is globally $k$-regular.
\kdow
\section{Open problems}
We finish our article stating challenging problems. 
\begin{enumerate}
\item Do there exist natural numbers $k,n,N$ such that there is a $k$-regular map $\c^n\rightarrow\c^N$, but there is no polynomial $k$-regular map?
\item Does there exist a polynomial $k$-regular map from $\c^3$ to $\c^{10}$ or from $\r^3$ to $\r^{10}$?
\item Can one improve currently best known lower bounds for $N$ under the assumption that the map is polynomial? 
\end{enumerate}
\section*{Acknowledgements}
Micha{\l}ek is grateful to Tadeusz Januszkiewicz for introducing him to the subject and to Jaros{\l}aw Buczy{\'n}ski and Joachim Jelisiejew for many interesting discussions on Gorenstein schemes. In fact a lot of ideas from this paper originated during these discussions. The work was realized as a part of Micha{\l}ek's PRIME DAAD fellowship - in particular Micha{\l}ek would like to thank Bernd Sturmfels and Klaus Altmann for great working conditions. Micha{\l}ek is also a member of AGATES group.
\bibliographystyle{amsalpha}
\bibliography{Xbib}
\end{document}